\documentclass[eqno]{siamltex704}
\usepackage{amsmath}
\usepackage{graphicx}
\usepackage{mathrsfs}
\usepackage{float}
\usepackage{amsfonts,amssymb}
\usepackage{dsfont}
\usepackage{pifont}
\usepackage{hyperref}
\usepackage{multirow}
\usepackage{color}
\usepackage{geometry}
\usepackage[comma, sort&compress, numbers]{natbib}
\newtheorem{remark}{Remark}[section]
\newtheorem{example}{Example}[section]

\newtheorem{algorithm}{Weak Galerkin Algorithm}

\newcommand{\bu}{{\bf u}}

\newcommand{\bw}{{\bf w}}

\newcommand{\be}{{\bf e}}
\newcommand{\bv}{{\mathbf v}}

\def\T{{\mathcal T}}

\def\pT{{\partial T}}
\def\l{{\langle}}
\def\r{{\rangle}}
\def\bbf{{\bf f}}

\def\bn{{\bf n}}

\def\3bar{{|\hspace{-.02in}|\hspace{-.02in}|}}

\newcommand{\bm}[1]{\mbox{\boldmath{$#1$}}}
\def\bPhi{\bm\Phi}
\def\bRT{\textbf{R}_T}

\title {A Locking-Free Weak Galerkin Finite Element Method for Linear Elasticity Problems}

\author{
Fuchang Huo\thanks{School of Mathematics, Jilin University, Changchun 130012, Jilin, China (huofc22@mails.jlu.edu.com)}
\and Ruishu
Wang\thanks{School of Mathematics, Jilin University, Changchun 130012, Jilin, China (wangrs\_math@jlu.edu.cn)} 
\and Yanqiu Wang\thanks{School of Mathematical Sciences, Nanjing Normal University, Nanjing 210023, Jiangsu, China (yqwang@njnu.edu.cn)} 
\and Ran Zhang\thanks{School of Mathematics, Jilin University, Changchun 130012, Jilin, China (zhangran@jlu.edu.cn)}
 }
\begin{document}

\maketitle

\begin{abstract}
In this paper, we introduce and analyze a lowest-order locking-free weak Galerkin (WG) finite element scheme for the grad-div formulation of linear elasticity problems. 
The scheme uses linear functions in the interior of mesh elements and constants on edges (2D) or faces (3D), respectively, to approximate the displacement.
An $H(div)$-conforming displacement reconstruction operator is employed to modify test functions in the right-hand side of the discrete form, in order to eliminate the dependence of the $Lam\acute{e}$ parameter $\lambda$ in error estimates, i.e., making the scheme locking-free. The method works without requiring $\lambda \|\nabla\cdot \mathbf{u}\|_1$ to be bounded. We prove optimal error estimates, independent of $\lambda$, in both the $H^1$-norm and the $L^2$-norm. Numerical experiments validate that the method is effective and locking-free.
\end{abstract}



\pagestyle{myheadings}
\section{Introduction}
\label{section:introduction}
Let $\Omega\subset\mathbb R^d\
(d = 2, 3)$ be an open bounded connected domain, and the boundary $\Gamma=\partial\Omega$ be Lipschitz continuous. We consider the linear elasticity  problem which seeks a displacement vector $\bu$ satisfying
\begin{eqnarray}\label{primal_model}
	-\mu\Delta\bu-(\lambda+\mu)\nabla(\nabla\cdot\bu)&=&\bbf,\qquad\, \text{in}\ \Omega,\\
	\bu&=&\textbf{0},\qquad \text{on}\ \Gamma, \label{bc2}
\end{eqnarray}
where $\bbf$ is the body force, $\mu$ and $\lambda$ are $Lam\acute{e}$ constants  satisfying $(\mu, \lambda) \in [\mu_1,\mu_2]\times(0,\infty)$, $0<\mu_1<\mu_2\ll\infty$ .

The weak formulation of the linear elasticity problems (\ref{primal_model})-(\ref{bc2})  can be written as: Finding $\bu\in [H_0^1(\Omega)]^d$ such that 
\begin{eqnarray}\label{model-weak}
	\mu(\nabla\bu,\nabla\bv)+(\lambda+\mu)(\nabla
	\cdot\bu,\nabla \cdot\bv)&=&(\bbf,\bv), \,\, \forall\, \bv\,\in
	[H_0^1(\Omega)]^d,
\end{eqnarray}
where $H^1(\Omega)$ and $H_0^1(\Omega)$ are defined by
\begin{align*}
	H^1(\Omega)&=\{v\in L^2(\Omega):\, \nabla v\in [L^2(\Omega)]^d\},\\
	H_0^1(\Omega)&=\{v\in H^1(\Omega):\, v|_{\Gamma}=0\}.
\end{align*}

In this paper, we use the usual notations for Sobolev
spaces \cite{ciarlet-fem}. For any open bounded  domain  $D\subset \mathbb{R}^d$, $(\cdot,\cdot)_{s,D}$, $|\cdot|_{s,D}$. and $\|\cdot\|_{s,D}$ represent the inner product, seminorm, and norm in the Sobolev space $H^s(D)$($s\ge 0$), respectively.

Elastic materials become incompressible or nearly incompressible when the $Lam\acute{e}$ constant  $\lambda$  goes to $\infty$. In this case,  the finite element solution often does not converge to the solution of the original problem, or the convergence  cannot reach optimal order. This is called the ``locking" phenomenon of elasticity problems \cite{b2008,Ciarlet,Brezzi-Fortin}. The reason for ``locking" phenomenon is that the finite element approximate error depends on $Lam\acute{e}$ constant $\lambda$. At present, existing methods to circumvent the locking phenomenon include mixed finite element  methods (MFEM)  \cite{abd1984,bbf2009,cgg2010,hs2007,mhs2009,vogelius}, nonconforming finite element methods (NC-FEM) \cite{babuska-suri,LeeC,zhangzhimin1997}, discontinuous Galerkin methods (DG) \cite{Daniele2013,larson,Wihler}, virtual element  methods (VEM) \cite{Virtualelastic2014,Virtualelastic2013}, and so on. In this paper, we consider using the weak Galerkin (WG) finite element method  to  solve  the linear elasticity problems, and introduce special techniques to make it locking-free.

The WG method is an efficient numerical method for solving partial differential equations (PDEs). It was first proposed by J. Wang and X. Ye in \cite{wysec2} for solving second-order elliptic problems. Recently, the WG method has attracted  many attentions and  achieved abundant results, such as: the multigrid method for WG \cite{CWWY15}, WG for the incompressible flow \cite{ZL18}, the post-processing technique for WG  \cite{WangZhangZhangZhang18}, the maximum principle of WG \cite{maximumwang2,WangYeZhaiZhang18} and so on. In addition, WG method has been successfully applied to biharmonic equations \cite{MuWangYeZhang14,ZhangZhai15},  Stokes equations \cite{wy1302,WangZhaiZhangS2016,wangwangliu2022},  Navier-Stokes equations \cite{HMY2018,LLC18,zzlw2018},  Maxwell's equations \cite{MLW14},  Brinkman equations \cite{MuWangYe14,WangZhaiZhang2016,ZhaiZhangMu16},  elasticity equations \cite{chenxie2016,Liu,lockingw,Yisongyang}, etc.

WG method for linear elasticity problems based on primal  formulations has been explored in \cite{Liu,LiuWang2023,lockingw,Yisongyang}. Among them, the methods proposed in \cite{LiuWang2023,lockingw} are locking-free because they are shown to be equivalent to locking-free displacement-pressure mixed formulations. On the other hand,  in papers \cite{Liu, Yisongyang}, lowest-order WG methods based on the local Raviart-Thomas spaces  are considered, which do not require a stabilization term.

In this paper, we construct a locking-free weak Galerkin finite element method for linear elasticity using the general polynomial space and  an  $H(div)$-conforming displacement reconstruction operator. In each element, we use  piecewise linear vector functions in the interior and piecewise constant functions on the boundary to approximate the displacement, respectively. 
In order to  make the scheme locking-free, we introduce a divergence-preserving displacement reconstruction operator  to modify  the test function in the  right-hand side  of the discrete problem. The approximate error is then  independent of the $Lam\acute{e}$ constant $\lambda$.  An important advantage of our method is that the term $\lambda \|\nabla\cdot \bu\|_1$ does not need to meet additional boundedness condition. Moreover, our numerical scheme only modifies the right side but does not change the stiffness matrix. Thus it is easy to implement.

The paper is organized as follows. Section \ref{Section:fem-algorithms} is devoted to constructing a weak Galerkin finite element scheme for linear elasticity problems.
In Section \ref{section:H1}, we present the optimal order error estimate in the $H^1$-norm. Section \ref{section:L2} is devoted to deriving the optimal order error estimate in the $L^2$-norm. Finally, in Section \ref{section:Numex}, we give  numerical results to  validate the accuracy and locking-free property of the algorithm.

\section{Numerical Algorithms}\label{Section:fem-algorithms}
Let ${\cal T}_h$ be a  simplicial partition of $\Omega\subset \mathbb{R}^d$ satisfying the shape regularity hypotheses in
\cite{wy1202}. For each element $T\in \T_h$, denote by $h_T$ the diameter of
$T$. Let $h=\max_{T\in \T_h} h_T$  be the mesh size of $\T_h$.
Denote by ${\cal E}_h$ the set of all edges (2D) or faces (3D) in
${\cal T}_h$ and  by ${\cal E}^0_h$
the set of all interior edges (2D) or faces (3D) in ${\cal T}_h$. Let $P_k(T)$ be the set of polynomials on $T$ of  degree no more than $k$.

We define the weak finite element space $V_h$ as
\begin{align*}
	V_h=\{\bv =\{\bv _0,\bv_b\}: \ \bv_0|_T \in [P_1(T)]^d, \bv_b |_e \in [P_0(e)]^d,\, \forall\, T\in \T_h, \, \forall\, e\in{\cal E}_h\}.
\end{align*}

Define a subspace of $V_h$ as
\begin{equation}\label{EQ:WFE-global}
	V^0_h=\{\bv =\{\bv _0,\bv_b\}\in V_h:\ \bv_b=\textbf{0} ~\text{on}~ \Gamma\}.
\end{equation}

For $\bv \in V_h$, its discrete weak divergence $\nabla_{w,0}\cdot\bv$ and  discrete weak gradient
$\nabla_{w,0}\bv$ are calculated  element by element on each  $T\in {\cal T}_h$  by
\begin{align*}
	(\nabla_{w,0}\cdot \bv)|_T=&\nabla_{w,0,T}\cdot
	(\bv|_T),\\
	(\nabla_{w,0}  \bv)|_T=&\nabla_{w,0,T} (\bv|_T),
\end{align*}
where  $\nabla_{w,0,T}\cdot \bv \in P_0(T)$ and $\nabla_{w,0,T} \bv \in [P_0(T)]^{d\times d}$ are defined by
\begin{equation}\label{weak-divergence}
	(\nabla_{w,0,T}\cdot \bv,\phi)_T=-(\bv_0,\nabla
	\phi)_T+\langle \bv_b\cdot \bn,
	\phi\rangle_{\partial T}=\langle \bv_b\cdot \bn,
	\phi\rangle_{\partial T},\quad \forall\,\phi\in P_0(T),
\end{equation}
\begin{equation}\label{weak-gradient}
	(\nabla_{w,0,T} \bv,\varphi)_T=-(\bv_0,\nabla \cdot
	\varphi)_T+\langle \bv_b,  \varphi \bn
	\rangle_{\partial T}=\langle \bv_b,  \varphi\bn
	\rangle_{\partial T},\quad\forall\, \varphi\in[P_0(T)]^{d\times d},
\end{equation}
respectively. Here $\bn$ is the unit outward normal direction on $\partial T$, $\langle \bv_b\cdot \bn,
\phi\rangle_{\partial T}=\int_{\partial T} \bv_b\cdot \bn
\phi\,ds$ and  $\langle \bv_b,\varphi \bn
\rangle_{\partial T}=\int_{\partial T} \bv_b\cdot (\varphi \bn) ds$ are usual $L^2$ inner products.

For simplicity, we abbreviate the notation $\nabla_{w,0}\cdot$ and $\nabla_{w,0}$ as $\nabla_{w}\cdot$ and $\nabla_{w}$ in the rest of the paper, respectively.
For each $e\in\mathcal{E}_h$, $Q_b$ represents the $L^2-$ projection operator onto the space $[P_0(e)]^d$.

Similar to \cite{MuStkoes2020,MuyezhangStkoes2021,wangMuStkoes2021,wangwangliu2022}, we introduce a displacement reconstruction operator $\mathbf{R}_h:V_h\to H(div;\Omega):=\{\bw\in [L^2(\Omega)]^d,\, \nabla\cdot \bw \in L^2(\Omega)\}$ as follows. On each $T\in \T_h$, define 
$$\mathbf{R}_h|_T=\bRT:V_h|_T\to RT_0(T),$$
where $RT_0(T)=[P_0(T)]^d+\textbf{x}P_0(T)$ is the lowest-order Raviart-Thomas space on $T$, satisfying
\begin{equation}\label{RT}
	\mathbf{R}_h(\bv)\cdot \bn_e = \bv_b\cdot \bn_e, \,\, \forall\, e\subset  \partial T.
\end{equation}
Here $\bn_e$ is the unit outward normal vector on $e\subset  \partial T$. The definition ensures that $\bRT(\bv)\in H(div;\Omega)$.

We introduce the following bilinear forms
\begin{eqnarray}\label{EQ:stabilizer}
	s(\bw,\bv)&=&\sum_{T\in\mathcal{T}_h}h_T^{-1}\langle Q_b\bw_0-\bw_b,Q_b\bv_0-\bv_b\rangle_{\partial T},
	\\
	a(\bw,\bv)&=&\mu\sum_{T\in\mathcal{T}_h}(\nabla_w\bw,\nabla_w\bv)_T
	+(\lambda+\mu)\sum_{T\in\mathcal{T}_h}(\nabla_w\cdot\bw,\nabla_w\cdot\bv)_T,\label{EQ:primalbilinear}
	\\
	a_s(\bw,\bv)&=&a(\bw,\bv)+s(\bw,\bv).\label{EQ:bilinearForm}
\end{eqnarray}

We then propose a new WG scheme in Algorithm \ref{algo-primal}, and compare it with the standard WG scheme in Algorithm \ref{algo-primal_old}.

\begin{algorithm}
	\label{algo-primal}
	A new WG scheme for weak formulation (\ref{model-weak}) is given by: find $\bu_h=\{\bu_0, \bu_b\}\in V_h^0$ such that
	\begin{eqnarray}\label{WGA_primal}
		a_s(\bu_h,\bv)=(\bbf, {\mathbf {R}}_h(\bv)), \qquad\forall\, \bv=\{\bv_0, \bv_b\}\in
		V_h^0.
	\end{eqnarray}
\end{algorithm}

\begin{algorithm}
	\label{algo-primal_old}
	A standard	WG scheme for weak formulation (\ref{model-weak}) is given by: find $\bu_h=\{\bu_0, \bu_b\}\in V_h^0$ such that
	\begin{eqnarray}\label{WGA_primal_old}
		a_s(\bu_h,\bv)=(\bbf,\bv_0), \qquad\forall\, \bv=\{\bv_0, \bv_b\}\in
		V_h^0.
	\end{eqnarray}
\end{algorithm}

\begin{theorem}
	The new WG scheme (\ref{WGA_primal}) admits a unique solution.
\end{theorem}

\begin{proof}
	For finite dimensional systems, we need only to prove the uniqueness of the solution.
	Let $\bu_h^{(j)}=\{\bu_0^{(j)}, \bu_b^{(j)}\}\in V_h^0,\
	j=1,2$ be two solutions of (\ref{WGA_primal}), we have 
	\begin{eqnarray*}
		a_s(\bu_h^{(j)},\bv)=(\bbf,\mathbf{R}_h(\bv)), \qquad\forall\, \bv=\{\bv_0,
		\bv_b\}\in V_h^0, \ j=1,2.
	\end{eqnarray*}
	Let $\bw=\bu_h^{(1)}-\bu_h^{(2)}$, we obtain $\bw\in V_h^0$ and
	\begin{eqnarray}\label{Uniqueness:001}
		a_s(\bw,\bv)=0, \qquad\forall \,\bv=\{\bv_0, \bv_b\}\in V_h^0.
	\end{eqnarray}
	Letting $\bv=\bw$ in (\ref{Uniqueness:001}) we get
	$$
	a_s(\bw,\bw) = 0.
	$$
	By using the definition of $a_s(\cdot,\cdot)$ we have
	\begin{eqnarray*}
		\sum_{T\in\mathcal{T}_h}(\mu\nabla_w\bw,\nabla_w\bw)_T+\sum_{T\in\mathcal{T}_h}((\lambda+\mu)\nabla_w\cdot\bw,\nabla_w\cdot\bw)_T\\
		+\sum_{T\in\mathcal{T}_h}h_T^{-1}\langle
		Q_b\bw_0-\bw_b,Q_b\bw_0-\bw_b\rangle_{\partial T}=0,
	\end{eqnarray*}
	which leads to
	\begin{eqnarray}
		\nabla_w\bw&=&\textbf{0}, \qquad \text{in}\  T,\label{EQ:July12:000}\\
		\nabla_w\cdot\bw&=&0, \qquad \text{in}\  T,\\
		Q_b\bw_0-\bw_b&=&\textbf{0},  \qquad \text{on}\  \pT.\label{EQ:July12:001}
	\end{eqnarray}
	It follows from the definition of the discrete weak gradient (\ref{weak-gradient}), we arrive at
	\begin{eqnarray*}
		0&=&(\nabla_w\bw,\tau)_T
		\\
		&=&-(\bw_{0},\nabla\cdot\tau)_T+\langle\bw_{b},\tau\bn\rangle_\pT
		\\
		&=&(\nabla\bw_{0},\tau)_T-\langle\bw_{0}-\bw_{b},\tau \bn\rangle_\pT
		\\
		&=&(\nabla\bw_{0},\tau)_T-\langle
		Q_b\bw_{0}-\bw_{b},\tau\bn\rangle_\pT ,\,\, \forall\tau\in [P_{0}(T)]^{d\times d}.
	\end{eqnarray*}
	Letting $\tau=\nabla\bw_{0}$ in the equation, using $Q_b\bw_{0}-\bw_{b}=\textbf{0}$, we have $\nabla\bw_0=\textbf{0}$ on each element
	$T$. It follows that $\bw_{0}=constant$ on each element $T$. Thus, with the facts that $Q_b\bw_{0}=\bw_{0}=\bw_{b}$ on $\pT$ and
	$\bw_{b}=\textbf{0}$ on $\Gamma$, we obtain $\bw_{0}\equiv \textbf{0}$ and $\bw_{b}\equiv \textbf{0}$ in $\Omega$.
	This shows that $\bu_h^{(1)} \equiv
	\bu_h^{(2)}$, hence we get the uniqueness and the existence of the solution.
	 
\end{proof}

\section{Error Estimate in Discrete $H^1$ Norm}\label{section:H1}
Define a semi-norm on $V_h$ as follows
\begin{equation}\label{EQ:norm}
	\3bar\bv\3bar =\left (\sum_{T\in {\cal T}_h}\|\nabla_w
	\bv\|_T^2+ h_T^{-1}\| Q_b\bv_0-\bv_b\|^2_{\partial
		T}\right )^{\frac{1}{2}},\quad \bv\in V_h.
\end{equation}

It is  easy to see  that $\3bar\cdot\3bar$ satisfies the following properties.
\begin{lemma}\label{lemmanorm}
	$\3bar\cdot\3bar$ is a norm in the weak finite element space $V^0_h$.
\end{lemma}

\begin{lemma}\label{lemmacoerv}
	Let $\alpha=\min\{\mu,1\}$, then
	\begin{equation}\label{EQ:coercivity}
		a_s(\bv,\bv)\geq \alpha\3bar \bv
		\3bar^2, \qquad \forall \ \bv\in V_h^0.
	\end{equation}
\end{lemma}

The following results have been proved in \cite{wy1302}.

\begin{lemma}\cite{wy1302}
	For any $\bv=\{\bv_{0},\bv_{b}\}\in V_h$, there exist $C>0$ such that
	\begin{equation}\label{EQ:v0}
		\sum_{T\in {\cal T}_h}\|\nabla\bv_{0}\|_T^2\leq C \3bar\bv\3bar^2.
	\end{equation}
	Furthermore, we have
	\begin{equation}\label{EQ:v0vb}
		\sum_{T\in {\cal T}_h}h_T^{-1}\|\bv_{0}-\bv_{b}\|_{\partial
			T}^2 \leq C \3bar\bv\3bar^2.
	\end{equation}
\end{lemma}

For each element $T\in {\cal T}_h$, $Q_0$ denotes the $L^2$
projection operator onto $[P_1(T)]^d$.
$Q_h \textbf{u}:= \{Q_0\textbf{u}, Q_b \textbf{u}\}$ represents the $L^2$
projection  operator onto $V_h$, where $Q_b$ is the $L^2$ projection onto $[P_0(e)]^d$.
Besides, denote by ${\cal Q}_h$ and $\textbf{Q}_h$ the $L^2$
projection operator onto $P_{0}(T)$ and $[P_{0}(T )]^{d\times d}$, respectively.

\begin{lemma}\cite{wy1302} \label{lemmaProjection} For any $\bv \in [H^1(\Omega)]^d$, the projection operators $Q_h$, ${\emph {\textbf{Q}} }_h$, and ${\cal Q}_h$satisfies the following properties:
	\begin{eqnarray}\label{EQ:projection_gradient}
		\nabla_w (Q_h \bv)&=&{\emph {\textbf{Q}} }_h (\nabla \bv),\\
		\nabla_w \cdot(Q_h \bv)&=&{\cal Q}_h (\nabla \cdot \bv).\label{EQ:projection_div}
	\end{eqnarray}
\end{lemma}

\begin{lemma} The displacement reconstruction operator ${\emph {\textbf {R}}}_h$ is divergence-preserving, i.e., for all $\bv=\{\bv_{0},\bv_{b}\}\in V_h$, it satisfies
	\begin{equation}\label{RTProp}
		\nabla\cdot  ({\emph {\textbf {R}}}_h(\bv))=\nabla_{w}\cdot\bv.
	\end{equation}
\end{lemma}
\begin{proof}
	For any $q\in P_0(T)$, we have $\nabla q=\textbf{0}$. Together with the definition of the weak divergence (\ref{weak-divergence}) and the definition of the reconstruction operator $\textbf{R}_h$ (\ref{RT}), we obtain
	\begin{align*}
		(\nabla\cdot (\bRT(\bv)),q)_T&=-(\bRT(\bv),\nabla q)_T + \langle \bRT(\bv) \cdot  \bn, q\rangle_{\partial T}\\
		&= \langle \bv_{b} \cdot  \bn, q\rangle_ {\partial T} \\
		&= (\nabla_w \cdot  \bv, q)_T.
	\end{align*}
	Next, using the facts that $\nabla\cdot (\bRT(\bv))\in P_0(T)$ and $\nabla_{w}\cdot\bv\in P_0(T)$, we immediately get (\ref{RTProp}). This completes the proof.
	 
\end{proof}

\begin{lemma} \cite{MuStkoes2020,wangMuStkoes2021} For any $\bv=\{\bv_{0},\bv_{b}\}\in V_h$, there exist $C>0$ such that
	\begin{equation}\label{RTEs}
		\sum_{T\in{\cal	 T}_h}\|{\emph {\textbf {R}}}_T(\bv)-\bv_{0}\|^2_T \leq C\sum_{T\in{\cal T}_h}h_T\|Q_b\bv_0-\bv_b\|^2_{\partial T}\leq Ch^2\3bar \bv\3bar^2.
	\end{equation}
\end{lemma}

\begin{lemma}\label{lemmaA1} \cite{wy1202} Let ${\cal T}_h$ be a shape regular partition \cite{wy1202} of $\,\Omega$. For any  $\bw\in [H^2 (\Omega)]^ d$, $\rho \in H^1 (\Omega)$, and $0 \leq m \leq 1$, we have
	\begin{align}\label{A1}
		\sum_{T\in{\cal
				T}_h}h_T^{2m}\|\bw-Q_0\bw\|^2_{T,m}&\leq  C
		h^4\|\bw\|^2_2,\\
		\sum_{T\in{\cal T}_h}h_T^{2m}\|
		\nabla\bw-{\emph {\textbf{Q}} }_h\nabla\bw\|^2_{T,m}&\leq
		Ch^2\|\bw\|^2_2,\label{A2}\\
		\sum_{T\in{\cal T}_h}h_T^{2m}\|\rho-{\cal Q}_h\rho\|^2_{T,m}&\leq
		Ch^2\|\rho\|^2_1.\label{A3}
	\end{align}
\end{lemma}

Let $\bu\in [H^1_0(\Omega)]^d$ be the exact solution of (\ref{model-weak}) and $\bu_h=	\{\bu_0, \bu_b\} \in V_h^0$ be the WG solution to (\ref{WGA_primal}), define the error
\begin{align}\label{EQ:errorfunction}
	\be_h&=\{\be_0,\be_b\}=\{Q_0\bu-\bu_0,
	Q_b\bu-\bu_b\}.
\end{align}

\begin{lemma}\label{LemErrorEq}The error function
	$\be_h$ satisfies the
	following error equation
	\begin{align}\label{EQ:errorequation}
		a_s(\be_h,\bv)&=
		\varphi_{\bu}(\bv), \qquad \forall \,\bv\in V_h^0,
	\end{align}
	where
	\begin{eqnarray*}
		\varphi_{\bu}(\bv) &=& \ell_\bu(\bv)-
		\theta_\bu(\bv)+ s(Q_h\bu,\bv),\\
		\ell_\bu(\bv) &=&\mu \sum_{T\in {\cal T}_h} \langle(\nabla\bu-{\emph {\textbf{Q}} }_h\nabla\bu)\bn,\bv_{0}-\bv_{b}\rangle_\pT,\\
		\theta_\bu(\bv) &=& \mu\sum_{T\in {\cal T}_h}(\Delta\bu,\bv_0-{\emph{\textbf{R}}}_T(\bv))_T.
	\end{eqnarray*}
\end{lemma}
\begin{proof}
	Using (\ref{EQ:projection_gradient}), (\ref{weak-gradient}), integration by parts, and the definition of $\textbf{Q}_h$, we have
	\begin{equation}\label{EQ:weakgradientQh}
		\begin{split}
			& (\nabla_w(Q_h\bu), \nabla_w\bv)_T\\
			=& ( \textbf{Q}_h \nabla\bu,\nabla_w\bv)_T\\
			=&  -(\bv_0, \nabla\cdot (\textbf{Q}_h \nabla\bu
			))_T+\langle \bv_b, (\textbf{Q}_h\nabla\bu)\bn\rangle_{\partial T}\\
			=&  ( \nabla \bv_0, \textbf{Q}_h\nabla\bu)_T-\langle
			\bv_0 - \bv_b, (\textbf{Q}_h\nabla\bu)\bn\rangle_{\partial T}\\
			=&  (\nabla\bv_0, \nabla\bu) _T-\langle \bv_0 - \bv_b
			, (\textbf{Q}_h\nabla\bu)\bn\rangle_{\partial T}.
		\end{split}
	\end{equation}
	Note that $\nabla\cdot\bu$ is continuous and $\bRT(\bv)$ has normal continuity across $e\in {\cal E}_h$, we get
	$$\sum_{T\in {\cal T}_h}\l\nabla\cdot\bu,\bRT(\bv)\cdot \bn\r_{\partial T}=0.$$
	Combining the above with (\ref{EQ:projection_div}), the properties of operator $\bRT$ in (\ref{RTProp}), and the definition of ${\cal Q}_h$, we obtain
	\begin{equation}\label{EQ:weakdivQh}
		\begin{split}
			&\sum_{T\in {\cal T}_h} (\nabla_w\cdot(Q_h\bu), \nabla_w\cdot\bv)_T\\
			=&\sum_{T\in {\cal T}_h} ({\cal Q}_h (\nabla\cdot\bu),\nabla_w\cdot\bv)_T\\
			=&\sum_{T\in {\cal T}_h}({\cal Q}_h (\nabla\cdot\bu),\nabla\cdot\bRT(\bv))_T\\
			=&\sum_{T\in {\cal T}_h}(\nabla\cdot\bu,\nabla\cdot\bRT(\bv))_T\\
			=&\sum_{T\in {\cal T}_h}(\nabla\cdot\bu,\nabla\cdot\bRT(\bv))_T-\sum_{T\in {\cal T}_h}\langle\nabla\cdot\bu,\bRT(\bv)\cdot \bn\rangle_{\partial T}\\
			=&-\sum_{T\in {\cal T}_h}(\nabla(\nabla\cdot\bu),\bRT(\bv))_T.
		\end{split}
	\end{equation}
	Next, testing (\ref{primal_model}) with $\bRT(\bv)$, we have
	\begin{equation}\label{EQ:f_RTv}
		\begin{split}
			&\sum_{T\in {\cal T}_h} (\bbf,\bRT(\bv))_T\\
			=&-\mu\sum_{T\in {\cal T}_h} (\Delta\bu,\bRT(\bv))_T-\sum_{T\in {\cal T}_h} (\lambda+\mu)(\nabla(\nabla\cdot\bu),\bRT(\bv))_T\\
			=&-\mu\sum_{T\in {\cal T}_h} (\Delta\bu,\bv_{0})_T+\mu\sum_{T\in {\cal T}_h} (\Delta\bu,\bv_{0}-\bRT(\bv))_T\\
			&-\sum_{T\in {\cal T}_h} (\lambda+\mu)(\nabla(\nabla\cdot\bu),\bRT(\bv))_T.
		\end{split}
	\end{equation}
	
	By using integration by parts, the fact that $\sum_{T\in{\cal T}_h}\langle\bv_{b},(\nabla\bu)\bn\rangle_\pT = 0$, and (\ref{EQ:weakgradientQh}), we get
	\begin{equation}\label{EQ:laplaceu_v0}
		\begin{split}
			& -\mu\sum_{T\in{\cal T}_h}(\Delta\bu,\bv_{0})_T\\
			=&\mu\sum_{T\in{\cal T}_h}(\nabla\bu,\nabla\bv_{0})_T-\mu\sum_{T\in{\cal T}_h}\langle\bv_{0},(\nabla\bu)\bn\rangle_\pT\\
			=&\mu\sum_{T\in{\cal T}_h}(\nabla\bu,\nabla\bv_{0})_T-\mu\sum_{T\in{\cal T}_h}\langle\bv_{0}-\bv_{b},(\nabla\bu)\bn\rangle_\pT\\
			=&\mu\sum_{T\in{\cal T}_h}(\nabla_w(Q_h\bu), \nabla_w\bv)_T+\mu\sum_{T\in{\cal T}_h}\langle \bv_0 - \bv_b,(\textbf{Q}_h\nabla\bu)\bn\rangle_{\partial T}\\
			&-\mu\sum_{T\in{\cal T}_h}\langle\bv_{0}-\bv_{b},(\nabla\bu)\bn\rangle_\pT\\
			=&\mu\sum_{T\in{\cal T}_h}(\nabla_w(Q_h\bu), \nabla_w\bv)_T-\mu\sum_{T\in{\cal T}_h}\langle \bv_0 - \bv_b,(\nabla\bu-\textbf{Q}_h\nabla\bu)\bn\rangle_{\partial T}.
		\end{split}
	\end{equation}
	From (\ref{EQ:weakdivQh}), (\ref{EQ:f_RTv}), and (\ref{EQ:laplaceu_v0}), we obtain
	\begin{equation*}
		\begin{split}
			&\mu\sum_{T\in{\cal T}_h}(\nabla_w(Q_h\bu), \nabla_w\bv)_T+(\lambda+\mu)\sum_{T\in{\cal T}_h}(\nabla_w\cdot(Q_h\bu), \nabla_w\cdot\bv)_T\\
			=& \sum_{T\in{\cal T}_h}(\bbf,\bRT(\bv))_T-\mu\sum_{T\in {\cal T}_h}(\Delta\bu,\bv_{0}-\bRT(\bv))_T\\
			&+\mu\sum_{T\in {\cal T}_h}\langle(\nabla\bu-\textbf{Q}_h\nabla\bu)\bn,\bv_{0}-\bv_{b}\rangle_\pT,
		\end{split}
	\end{equation*}
	which implies that
	\begin{equation}\label{EQ:asQhuv}
		\begin{split}
			a_s(Q_h\bu,\bv)=\sum_{T\in{\cal T}_h}(\bbf,\bRT(\bv))_T+\ell_\bu(\bv)-\theta_\bu(\bv)+s(Q_h\bu,\bv),
		\end{split}
	\end{equation}
	where
	\begin{eqnarray*}
		\ell_\bu(\bv) &=&\mu \sum_{T\in {\cal T}_h}\langle(\nabla\bu-\textbf{Q}_h\nabla\bu)\bn,\bv_{0}-\bv_{b}\rangle_\pT,\\
		\theta_\bu(\bv) &=&\mu \sum_{T\in {\cal T}_h}(\Delta\bu,\bv_0-\bRT(\bv))_T.
	\end{eqnarray*}
	Subtracting (\ref{WGA_primal}) from (\ref{EQ:asQhuv}), we get the error equation (\ref{EQ:errorequation}). This completes the proof.
	 
\end{proof}

Next, we estimate the terms in the right-hand side of the error equation (\ref{EQ:errorequation}).

\begin{lemma}\label{LemProEstimate} Let $\bw\in [H^2 (\Omega)]^ d$ and $\bv \in V_h$, we have
	\begin{align}\label{A4}
		|s(Q_h\bw,\bv)|\leq &
		Ch\|\bw\|_{2}\3bar\bv\3bar,\\
		|\ell_\bw(\bv)|\leq&
		Ch\|\bw\|_{2}\3bar\bv\3bar,\label{A5}\\
		|\theta_\bw(\bv)|\leq &
		Ch\|\bw \|_2\3bar\bv\3bar,\label{A6}
	\end{align}
	where the general constant $C>0$ is independent of $\lambda$.
\end{lemma}

\begin{proof}
	It follows from (\ref{EQ:stabilizer}), the
	Cauchy-Schwarz inequality, the trace inequality, and the estimate
	(\ref{A1}) that
	\begin{equation*}
		\begin{split}
			|s(Q_h\bw,\bv)|=&\left|\sum_{T\in {\cal
					T}_h}h_T^{-1}\langle Q_bQ_0\bw-Q_b\bw, Q_b\bv_0-
			\bv_b\rangle_{\partial
				T}\right|\\
			=&\left|\sum_{T\in {\cal T}_h}h_T^{-1}\langle
			Q_0\bw - \bw,Q_b\bv_0-
			\bv_b\rangle_{\partial
				T}\right|\\
			\leq & \left (\sum_{T\in {\cal T}_h}h_T^{-1}\|Q_0\bw-
			\bw\|_{\partial T}^2\right )^{\frac{1}{2}}\left (\sum_{T\in {\cal
					T}_h}h_T^{-1}\|Q_b\bv_0- \bv_b\|_{\partial
				T}^2\right )^{\frac{1}{2}}\\
			\leq & C\left (\sum_{T\in {\cal T}_h}h_T^{-2}\|Q_0\bw- \bw\|_{T}^2 +
			\|\nabla(Q_0\bw- \bw)\|_{T}^2\right )^{\frac{1}{2}} \3bar \bv \3bar\\
			\leq & Ch\|\bw\|_2\3bar \bv \3bar.
		\end{split}
	\end{equation*}
	We then prove (\ref{A5}). Using the Cauchy-Schwarz inequality, the trace inequality, the estimates (\ref{EQ:v0vb}) and (\ref{A2}), we get
	\begin{equation*}
		\begin{split}
			|\ell_\bw(\bv)|=&\left| \mu\sum_{T\in{\cal
					T}_h}\langle( \nabla\bw-\textbf{Q}_h\nabla
			\bw)\bn,\bv_0-\bv_b\rangle_{\partial T}\right|\\
			\leq & C \left (\sum_{T\in{\cal T}_h}h_T \|
			\nabla\bw-\textbf{Q}_h\nabla\bw\|^2_{\partial
				T}\right )^{ \frac{1}{2}} \left (\sum_{T\in{\cal T}_h}h_T^{-1}\|
			\bv_0-\bv_b\|^2_{\partial T}\right )^{ \frac{1}{2}}\\
			\leq & C h\|\bw\|_2\3bar \bv \3bar.
		\end{split}
	\end{equation*}
	Finally, by the Cauchy-Schwarz inequality and the estimate (\ref{RTProp}), we arrive at
	\begin{equation*}
		\begin{split}
			|\theta_\bw(\textbf{v})|&=\left| \mu\sum_{T\in {\cal T}_h}(\Delta\bw,\bv_0-\bRT(\bv))_T \right|\\
			&\leq C\|\bw\|_2 \left (\sum_{T\in{\cal T}_h}\|\bv_0-\bRT(\bv)\|^2_{T}\right )^{ \frac{1}{2}} \\
			&\leq  C h\|\bw\|_2\3bar\textbf{v}\3bar.
		\end{split}
	\end{equation*}
	 
\end{proof}

\begin{theorem}\label{ehH1}
	Let $\bu \in
	[H^2(\Omega)]^d $  be the solution of (\ref{model-weak}) and $\bu_h \in V_h $ be the solution of the WG scheme (\ref{WGA_primal}), we have
	\begin{equation}\label{th1}
		\3barQ_h\bu-\bu_h\3bar \leq Ch\|\bu\|_2,
	\end{equation}
	where $C>0$ is a  constant independent of $\lambda$ or the mesh size $h$.
\end{theorem}
\begin{proof}
	By letting $\bv=\be_h$ in error equation
	(\ref{EQ:errorequation}), we arrive at
	$$
	a_s(\textbf{e}_h,\textbf{e}_h)=\varphi_{\textbf{u}}(\textbf{e}_h).
	$$
	Using Lemma \ref{LemProEstimate} and Lemma \ref{lemmacoerv}, we obtain
	\begin{equation*}
		\alpha\3bar \textbf{e}_h\3bar^2 \leq a_s(\textbf{e}_h,\textbf{e}_h) \leq
		Ch\|\bu\|_2\3bar \textbf{e}_h\3bar,
	\end{equation*}
	which implies that
	\begin{eqnarray}\label{eh}
		\3bar\textbf{e}_h\3bar\leq
		Ch\|\bu\|_2.
	\end{eqnarray}
	This completes the proof.
	 
\end{proof}

\begin{remark}
	Based on the findings presented in Theorem \ref{ehH1}, it is evident that the displacement error measured in the $H^1$-norm is independent of the $Lam\acute{e}$ constant $\lambda$.
	This indicates that the new WG algorithm \ref{algo-primal} is locking-free.
\end{remark}

\section{ Error Estimate in $L^2$ norm}\label{section:L2}
To derive an $L^2$ norm error estimate for the new WG Algorithm \ref{algo-primal}, we consider the duality problem: Seeking
$\bPhi$ such that
\begin{eqnarray}\label{dualityEq}
	-\mu\Delta\bPhi-(\lambda+\mu)\nabla(\nabla\cdot\bPhi)&=&\textbf{e}_0,\qquad\text{in}\ \Omega,\\
	\bPhi&=&\textbf{0}, \ \qquad\text{on} \ \Gamma.\label{bdy}
\end{eqnarray}
Assume that the dual problem (\ref{dualityEq})-(\ref{bdy}) satisfies the following 
regularity estimate:
\begin{equation}\label{dualityEqregularity}
	\mu\|\bPhi\|_2+ (\lambda+\mu)\|\nabla\cdot\bPhi\|_1\leq C\|\textbf{e}_0 \|,
\end{equation}
where the positive constant $C$ is independent of $\lambda$. The above regularity assumption is reasonable as stated in Section 11.2 of \cite{b2008}.

\begin{theorem}\label{THL2} Let $\bu \in [H ^2 (\Omega)]^d$ be the solution of (\ref{primal_model}), and $\bu_h \in V_h$ be the WG solution arising from the WG scheme (\ref{WGA_primal}). Assume  $\bbf \in [H^1(\Omega)]^d$. There exists a positive constant $C$ independent of $\lambda$, such that
	\begin{equation}\label{the0}
		\|Q_0\bu-\bu_0\|\leq Ch^2(\|\bu\|_2+\|\bbf\|_1).
	\end{equation}
\end{theorem}

\begin{proof}
	The proof is divided into five steps.
	
	Step 1: 
	Testing (\ref{dualityEq}) with $\textbf{e}_0$, we get
	\begin{equation}\label{e0st1}
		\|\be_0\|^2=-\mu\sum_{T\in {\cal T}_h}(\Delta\bPhi,\be_0)_T-(\lambda+\mu)\sum_{T\in {\cal T}_h}(\nabla(\nabla\cdot\bPhi),\be_0)_T=I_1+I_2.
	\end{equation}
	
	Step 2: 
	By applying (\ref{EQ:laplaceu_v0}), (\ref{EQ:projection_gradient}), the definitions of $\textbf{e}_h$ and $\textbf{Q}_h$, we have
	\begin{equation}\label{e0st2}
		\begin{split}
			&I_1=-\mu\sum_{T\in {\cal T}_h}(\Delta\bPhi,\be_0)_T\\
			=&\mu\sum_{T\in{\cal T}_h}(\nabla_w(Q_h\bPhi), \nabla_w\textbf{e}_h)_T-\mu\sum_{T\in{\cal T}_h}\langle \be_0 - \be_b,(\nabla\bPhi-\textbf{Q}_h\nabla\bPhi)\bn\rangle_{\partial T}\\
			=&\mu\sum_{T\in{\cal T}_h}(\textbf{Q}_h\nabla\bPhi, \nabla_w(Q_h\bu)-\nabla_w\bu_h)_T-\mu\sum_{T\in{\cal T}_h}\langle \be_0 - \be_b,(\nabla\bPhi-\textbf{Q}_h\nabla\bPhi)\bn\rangle_{\partial T}\\
			=&\mu\sum_{T\in{\cal T}_h}(\textbf{Q}_h\nabla\bPhi, \textbf{Q}_h\nabla\bu-\nabla_w\bu_h)_T-\mu\sum_{T\in{\cal T}_h}\langle \be_0 - \be_b,(\nabla\bPhi-\textbf{Q}_h\nabla\bPhi)\bn\rangle_{\partial T}\\
			=&\mu\sum_{T\in{\cal T}_h}(\textbf{Q}_h\nabla\bPhi, \nabla\bu-\nabla_w\bu_h)_T-\mu\sum_{T\in{\cal T}_h}\langle \be_0 - \be_b,(\nabla\bPhi-\textbf{Q}_h\nabla\bPhi)\bn\rangle_{\partial T}\\
			=&\mu\sum_{T\in{\cal T}_h}(\textbf{Q}_h\nabla\bPhi-\nabla\bPhi, \nabla\bu-\nabla_w\bu_h)_T+\mu\sum_{T\in{\cal T}_h}(\nabla\bPhi, \nabla\bu-\nabla_w\bu_h)_T\\
			&-\mu\sum_{T\in{\cal T}_h}\langle \be_0 - \be_b,(\nabla\bPhi-\textbf{Q}_h\nabla\bPhi)\bn\rangle_{\partial T}\\
			=&A_1+A_2-\ell_{\bPhi}(\textbf{e}_h).
		\end{split}
	\end{equation}
	
	Using (\ref{EQ:projection_gradient}) and the fact that $\sum_{T\in{\cal T}_h}(\nabla\bPhi-\textbf{Q}_h\nabla\bPhi, \textbf{Q}_h\nabla\bu)_T=0$, the second item $A_2$ can be rewritten as
	\begin{align}\label{e0st21}
		\begin{split}
			&A_2=\mu\sum_{T\in{\cal T}_h}(\nabla\bPhi, \nabla\bu-\nabla_w\bu_h)_T \\
			=&\mu\sum_{T\in{\cal T}_h}(\nabla\bPhi, \nabla\bu)_T-\mu\sum_{T\in{\cal T}_h}(\nabla_w(Q_h\bPhi), \nabla_w\bu_h)_T 
			-\mu\sum_{T\in{\cal T}_h}(\nabla\bPhi-\textbf{Q}_h\nabla\bPhi, \nabla_w\bu_h)_T \\ 
			=&\mu\sum_{T\in{\cal T}_h}(\nabla\bPhi, \nabla\bu)_T-\mu\sum_{T\in{\cal T}_h}(\nabla_w(Q_h\bPhi), \nabla_w\bu_h)_T \\
			&-\mu\sum_{T\in{\cal T}_h}(\nabla\bPhi-\textbf{Q}_h\nabla\bPhi, \nabla_w\bu_h)_T
			+\mu\sum_{T\in{\cal T}_h}(\nabla\bPhi-\textbf{Q}_h\nabla\bPhi, \textbf{Q}_h\nabla\bu)_T \\
			=&\mu\sum_{T\in{\cal T}_h}(\nabla\bPhi, \nabla\bu)_T-\mu\sum_{T\in{\cal T}_h}(\nabla\bPhi-\textbf{Q}_h\nabla\bPhi, \nabla_w\bu_h-\nabla\bu)_T \\
			&-\mu\sum_{T\in{\cal T}_h}(\nabla_w(Q_h\bPhi), \nabla_w\bu_h)_T
			-\mu\sum_{T\in{\cal T}_h}(\nabla\bPhi-\textbf{Q}_h\nabla\bPhi, \nabla\bu-\textbf{Q}_h\nabla\bu)_T.
		\end{split}
	\end{align}
	
	By substituting (\ref{e0st21}) into (\ref{e0st2}), we arrive at
	\begin{equation}\label{e0st22}
		\begin{split}
			I_1=&-\mu\sum_{T\in {\cal T}_h}(\Delta\bPhi,\be_0)_T\\
			=&\mu\sum_{T\in{\cal T}_h}(\nabla\bPhi, \nabla\bu)_T-\mu\sum_{T\in{\cal T}_h}(\nabla\bPhi-\textbf{Q}_h\nabla\bPhi, \nabla\bu-\textbf{Q}_h\nabla\bu)_T\\
			&-\mu\sum_{T\in{\cal T}_h}(\nabla_w(Q_h\bPhi), \nabla_w\bu_h)_T-\mu\sum_{T\in{\cal T}_h}\langle \be_0 - \be_b,(\nabla\bPhi-\textbf{Q}_h\nabla\bPhi)\cdot\bn\rangle_{\partial T}.
		\end{split}
	\end{equation}
	
	Step 3: 
	By using integration by parts, the definition of ${\cal Q}_h$, (\ref{weak-divergence}), and the fact  that $\sum_{T\in{\cal T}_h}\langle\be_{b},(\nabla\cdot\bPhi)\bn\rangle_\pT = 0$, we have
	\begin{align}\label{e0st31}
		\begin{split}
		&-\sum_{T\in {\cal T}_h}(\nabla(\nabla\cdot\bPhi),\be_0)_T \\
		=&\sum_{T\in {\cal T}_h}(\nabla\cdot\bPhi,\nabla\cdot\be_0)_T-\sum_{T\in {\cal T}_h}\l(\nabla\cdot\bPhi)\bn,\be_0\r_{\partial T} \\
		=&\sum_{T\in {\cal T}_h}({\cal Q}_h\nabla\cdot\bPhi,\nabla\cdot\be_0)_T-\sum_{T\in {\cal T}_h}\l(\nabla\cdot\bPhi)\bn,\be_0\r_{\partial T} \\
		=&-\sum_{T\in {\cal T}_h}\l(\nabla\cdot\bPhi)\bn,\be_0\r_{\partial T}+\sum_{T\in {\cal T}_h}\l({\cal Q}_h\nabla\cdot\bPhi)\bn,\be_0\r_{\partial T} 
		-\sum_{T\in {\cal T}_h}(\nabla({\cal Q}_h\nabla\cdot\bPhi),\be_0)_T \\
		=&-\sum_{T\in {\cal T}_h}\l(\nabla\cdot\bPhi)\bn,\be_0\r_{\partial T}+\sum_{T\in {\cal T}_h}\l({\cal Q}_h\nabla\cdot\bPhi)\bn,\be_0-\be_b\r_{\partial T} 
		+\sum_{T\in {\cal T}_h}({\cal Q}_h\nabla\cdot\bPhi,\nabla_w\cdot \be_h)_T \\
		=&-\sum_{T\in {\cal T}_h}\l(\nabla\cdot\bPhi)\bn,\be_0-\be_b\r_{\partial T}+\sum_{T\in {\cal T}_h}\l({\cal Q}_h\nabla\cdot\bPhi)\bn,\be_0-\be_b\r_{\partial T}
		+\sum_{T\in {\cal T}_h}({\cal Q}_h\nabla\cdot\bPhi,\nabla_w\cdot \be_h)_T \\
		=&-\sum_{T\in {\cal T}_h}\l(\nabla\cdot\bPhi-{\cal Q}_h\nabla\cdot\bPhi)\bn,\be_0-\be_b\r_{\partial T}+\sum_{T\in {\cal T}_h}({\cal Q}_h\nabla\cdot\bPhi,\nabla_w\cdot \be_h)_T.
		\end{split}
	\end{align}
	
	For the second term $\sum_{T\in {\cal T}_h}({\cal Q}_h\nabla\cdot\bPhi,\nabla_w\cdot \be_h)_T$, it follows from the definition of $\textbf{e}_h$, (\ref{EQ:projection_div}), and the definition of ${\cal Q}_h$ that
	\begin{align}\label{e0st32}
		\begin{split}
			&\sum_{T\in {\cal T}_h}({\cal Q}_h\nabla\cdot\bPhi,\nabla_w\cdot \be_h)_T \\
			=&\sum_{T\in {\cal T}_h}({\cal Q}_h\nabla\cdot\bPhi,\nabla_w\cdot (Q_h\bu-\bu_h))_T \\
			=&\sum_{T\in {\cal T}_h}({\cal Q}_h\nabla\cdot\bPhi,\nabla\cdot\bu-\nabla_w\cdot\bu_h)_T \\
			=&\sum_{T\in {\cal T}_h}({\cal Q}_h\nabla\cdot\bPhi-\nabla\cdot\bPhi,\nabla\cdot\bu-\nabla_w\cdot\bu_h)_T+\sum_{T\in {\cal T}_h}(\nabla\cdot\bPhi,\nabla\cdot\bu-\nabla_w\cdot\bu_h)_T.		
		\end{split}
	\end{align}
	
	Reorganizing  the second term in (\ref{e0st32}),  we have
	\begin{equation}\label{e0st33}
		\begin{aligned}
			&\sum_{T\in {\cal T}_h}(\nabla\cdot\bPhi,\nabla\cdot\bu-\nabla_w\cdot\bu_h)_T\\
			=&\sum_{T\in {\cal T}_h}(\nabla\cdot\bPhi,\nabla\cdot\bu)_T-\sum_{T\in {\cal T}_h}(\nabla_w\cdot(Q_h\bPhi),\nabla_w\cdot\bu_h)_T\\
			&-\sum_{T\in {\cal T}_h}(\nabla\cdot\bPhi-{\cal Q}_h\nabla \cdot\bPhi,\nabla_w\cdot\bu_h)_T\\
			=&\sum_{T\in {\cal T}_h}(\nabla\cdot\bPhi,\nabla\cdot\bu)_T-\sum_{T\in {\cal T}_h}(\nabla\cdot\bPhi-{\cal Q}_h\nabla \cdot\bPhi,\nabla_w\cdot\bu_h-\nabla\cdot\bu)_T\\
			&-\sum_{T\in {\cal T}_h}(\nabla_w\cdot(Q_h\bPhi),\nabla_w\cdot\bu_h)_T-\sum_{T\in {\cal T}_h}(\nabla\cdot\bPhi-{\cal Q}_h\nabla \cdot\bPhi,\nabla\cdot\bu-{\cal Q}_h\nabla \cdot\bu)_T.
		\end{aligned}
	\end{equation}
	
	According to (\ref{e0st31}), (\ref{e0st32}), and (\ref{e0st33}), we obtain
	\begin{align}
		I_2=&-(\lambda+\mu)\sum_{T\in {\cal T}_h}(\nabla(\nabla\cdot\bPhi),\be_0)_T \nonumber\\
		=&-(\lambda+\mu)\sum_{T\in {\cal T}_h}\l(\nabla\cdot\bPhi-{\cal Q}_h\nabla\cdot\bPhi)\bn,\be_0-\be_b\r_{\partial T} \nonumber\\
		&+(\lambda+\mu)\sum_{T\in {\cal T}_h}({\cal Q}_h\nabla\cdot\bPhi,\nabla_w\cdot \be_h)_T \nonumber\\
		=&-(\lambda+\mu)\sum_{T\in {\cal T}_h}\l(\nabla\cdot\bPhi-{\cal Q}_h\nabla\cdot\bPhi)\bn,\be_0-\be_b\r_{\partial T} \label{e0st34} \\
		&+(\lambda+\mu)\sum_{T\in {\cal T}_h}(\nabla\cdot\bPhi,\nabla\cdot\bu)_T \nonumber\\
		&-(\lambda+\mu)\sum_{T\in {\cal T}_h}(\nabla_w\cdot(Q_h\bPhi),\nabla_w\cdot\bu_h)_T \nonumber\\
		&-(\lambda+\mu)\sum_{T\in {\cal T}_h}(\nabla\cdot\bPhi-{\cal Q}_h\nabla \cdot\bPhi,\nabla\cdot\bu-{\cal Q}_h\nabla \cdot\bu)_T \nonumber.
	\end{align}
	Step 4: Testing (\ref{primal_model}) by using $\bPhi$, together with the fact that $\bPhi=\textbf{0}$ on $\Gamma$, we obtain
	\begin{equation}\label{e0st41}
		\mu(\nabla\bu,\nabla\bPhi)+(\lambda+\mu)(\nabla\cdot\bu,\nabla\cdot\bPhi)=(\bbf,\bPhi).
	\end{equation}
	From the WG scheme (\ref{WGA_primal}), we have
	\begin{equation*}
		\begin{aligned}
			\mu\sum_{T\in {\cal T}_h}(\nabla_w\bu_h,\nabla_w(Q_h\bPhi))_T+(\lambda+\mu)\sum_{T\in {\cal T}_h}(\nabla_w\cdot\bu_h,\nabla_w\cdot(Q_h\bPhi))_T
			+s(\bu_h,Q_h\bPhi)=(\bbf, {\textbf{R}}_h(Q_h\bPhi)),
		\end{aligned}
	\end{equation*}
	which leads to
	\begin{equation}\label{e0st42}
		\begin{aligned}
			&\mu\sum_{T\in {\cal T}_h}(\nabla_w\bu_h,\nabla_w(Q_h\bPhi))_T+(\lambda+\mu)\sum_{T\in {\cal T}_h}(\nabla_w\cdot\bu_h,\nabla_w\cdot(Q_h\bPhi))_T\\
			=&(\bbf, {\textbf{R}}_h(Q_h\bPhi))-s(\bu_h,Q_h\bPhi)\\
			=&(\bbf, {\textbf{R}_h}(Q_h\bPhi))-s(Q_h\bu,Q_h\bPhi)+s(\be_h,Q_h\bPhi).
		\end{aligned}
	\end{equation}
	
	Step 5: Combining (\ref{e0st1}), (\ref{e0st22}), (\ref{e0st34}), (\ref{e0st41}), and (\ref{e0st42}), we arrive at
	\begin{equation}\label{e0st51}
		\begin{aligned}
			\|\be_0\|^2=&-\mu\sum_{T\in {\cal T}_h}(\Delta\bPhi,\be_0)_T-(\lambda+\mu)\sum_{T\in {\cal T}_h}(\nabla(\nabla\cdot\bPhi),\be_0)_T\\
			=&-\mu\sum_{T\in{\cal T}_h}\langle \be_0 - \be_b,(\nabla\bPhi-\textbf{Q}_h\nabla\bPhi)\bn\rangle_{\partial T}\\
			&-(\lambda+\mu)\sum_{T\in {\cal T}_h}\l(\nabla\cdot\bPhi-{\cal Q}_h\nabla\cdot\bPhi)\bn,\be_0-\be_b\r_{\partial T}\\
			&-\mu\sum_{T\in{\cal T}_h}(\nabla\bPhi-\textbf{Q}_h\nabla\bPhi, \nabla\bu-\textbf{Q}_h\nabla\bu)_T\\
			&-(\lambda+\mu)\sum_{T\in {\cal T}_h}(\nabla\cdot\bPhi-{\cal Q}_h\nabla \cdot\bPhi,\nabla\cdot\bu-{\cal Q}_h\nabla \cdot\bu)_T\\
			&+(\bbf,\bPhi)-(\bbf, {\textbf{R}_h}(Q_h\bPhi))+s(Q_h\bu,Q_h\bPhi)-s(\be_h,Q_h\bPhi).
		\end{aligned}	
	\end{equation}

	Now,  we estimate the  eight terms on the right-hand side of (\ref{e0st51}).
	
	(i) According to Lemma \ref{LemProEstimate} and Theorem \ref{ehH1}, it follows that
	\begin{equation}\label{e0st52}
		\left|\mu\sum_{T\in{\cal T}_h}\langle \be_0 - \be_b,(\nabla\bPhi-\textbf{Q}_h\nabla\bPhi) \bn\rangle_{\partial T}\right|\leq Ch\|\bPhi\|_2\3bar\be_h\3bar\leq Ch^2\|\bPhi\|_2\|\bu\|_2,
	\end{equation}
	\begin{equation}\label{e0st53}
		|s(\be_h,Q_h\bPhi)|\leq Ch\|\bPhi\|_2\3bar\be_h\3bar\leq Ch^2\|\bPhi\|_2\|\bu\|_2.
	\end{equation}
	
	(ii) It  follows from the Cauchy-Schwarz inequality, the trace inequality, the estimates (\ref{EQ:v0vb}) and (\ref{A3}), and Theorem \ref{ehH1} that
	\begin{equation}\label{e0st54}
		\begin{split}
			&\left|(\lambda+\mu)\sum_{T\in{\cal T}_h} \langle
			\nabla\cdot\bPhi- {\cal Q}_h(\nabla\cdot\bPhi) ,
			({\textbf{e}_0-\textbf{e}_b})\cdot\bn \rangle_{\partial T}  \right|\\
			&\leq C(\lambda+\mu) \left (\sum_{T\in{\cal T}_h}h_T \| \nabla\cdot\bPhi- {\cal Q}_h(\nabla\cdot\bPhi) \|^2_{\partial
				T}\right )^{ \frac{1}{2}} \left (\sum_{T\in{\cal T}_h}h_T^{-1}\|
			(\textbf{e}_0-\textbf{e}_b)\cdot\bn\|^2_{\partial T}\right )^{ \frac{1}{2}}
			\\
			&\leq  C (\lambda+\mu) h\| \nabla\cdot\bPhi\|_1\3bar\textbf{e}_h\3bar\\
			&\leq  C (\lambda+\mu) h^2\| \nabla\cdot\bPhi\|_1\|\bu\|_2.
		\end{split}
	\end{equation}
	
	In a same way, from the Cauchy-Schwarz inequality, the trace inequality, and the estimate (\ref{A1}), we arrive at
	\begin{equation}\label{e0st55}
		\begin{split}
			|s(Q_h\bu, Q_h\bPhi)|=&\left|\sum_{T\in {\cal T}_h}h_T^{-1}\langle
			Q_b(Q_0\bu)-Q_b\bu,Q_b(Q_0\bPhi)-Q_b\bPhi\rangle_{\partial
				T}\right|\\
			\leq & \left (\sum_{T\in {\cal T}_h}
			h_T^{-1}\|Q_0\bu-\bu\|^2_{\partial
				T}\right )^{\frac{1}{2}}\left (\sum_{T\in {\cal T}_h}
			h_T^{-1}\|Q_0\bPhi-\bPhi\|^2_{\partial T}\right )^{\frac{1}{2}}\\
			\leq & Ch^{2}\|\bu\|_{2}\|\bPhi\|_{2}.
		\end{split}
	\end{equation}

	(iii) From the Cauchy-Schwarz inequality and the estimate (\ref{A2}), we get
	\begin{equation}\label{e0st56}
		\begin{split}
			&\left| \mu\sum_{T\in{\cal T}_h}(\nabla\bPhi-\textbf{Q}_h\nabla\bPhi, \nabla\bu-\textbf{Q}_h\nabla\bu)_T \right|\\
			&\leq C \left (\sum_{T\in{\cal T}_h} \| \nabla\bPhi-\textbf{Q}_h\nabla\bPhi \|^2_{
				T}\right )^{ \frac{1}{2}} \left (\sum_{T\in{\cal T}_h} \| \nabla\bu-\textbf{Q}_h\nabla\bu \|^2_{
				T}\right )^{ \frac{1}{2}} \\
			&\leq Ch^2\|\bPhi\|_2\|\bu\|_2.
		\end{split}
	\end{equation}

	Similarly, using the Cauchy-Schwarz inequality and the estimate (\ref{A3}), we obtain
	\begin{equation}\label{e0st57}
		\begin{split}
			&\left| (\lambda+\mu)\sum_{T\in {\cal T}_h}(\nabla\cdot\bPhi-{\cal Q}_h\nabla \cdot\bPhi,\nabla\cdot\bu-{\cal Q}_h\nabla \cdot\bu)_T \right|\\
			&\leq C (\lambda+\mu)\left (\sum_{T\in{\cal T}_h} \| \nabla\cdot\bPhi-{\cal Q}_h\nabla \cdot\bPhi \|^2_{
				T}\right )^{ \frac{1}{2}} \left (\sum_{T\in{\cal T}_h} \| \nabla\cdot\bu-{\cal Q}_h\nabla \cdot\bu \|^2_{
				T}\right )^{ \frac{1}{2}} \\
			&\leq C(\lambda+\mu)h^2\|\nabla\cdot\bPhi\|_1\|\nabla\cdot\bu\|_1,\\
			&\leq C(\lambda+\mu)h^2\|\nabla\cdot\bPhi\|_1\|\bu\|_2.
		\end{split}
	\end{equation}
	
	(iv) We then estimate $(\bbf,\bPhi)-(\bbf, {\textbf R}_h(Q_h\bPhi))$. First, let $ Q_T^0$ be the  $L^2$-orthogonal projection operator onto the space $[P_0(T)]^d$. By the definition of $\bRT$, we have
	\begin{equation*}
		\sum_{T\in{\cal T}_h}(Q_T^0\bbf,\bPhi-\bRT(Q_h\bPhi))_T=0.
	\end{equation*}
	
	Thus, from the Cauchy-Schwarz inequality, the estimates (\ref{A1}) and (\ref{RTEs}), we obtain
	\begin{equation}\label{e0st59}
		\begin{split}
			&\left|(\bbf,\bPhi)-(\bbf, {\mathbf{R}}_h(Q_h\bPhi))\right|\\
			&=\left|\sum_{T\in{\cal T}_h}(\bbf-Q_T^0\bbf,\bPhi-\bRT(Q_h\bPhi))\right|\\
			&= \left (\sum_{T\in{\cal T}_h}\|\bbf-Q_T^0\bbf\|^2_T\right )^{ \frac{1}{2}} \left (\sum_{T\in{\cal T}_h}\|\bPhi-\bRT(Q_h\bPhi)\|^2_T\right )^{ \frac{1}{2}}\\
			&\leq Ch\|\bbf\|_1h\|\bPhi\|_1,\\
			&\leq Ch^2\|\bbf\|_1\|\bPhi\|_2.
		\end{split}
	\end{equation}
	
	Combining the above and the regularity assumption (\ref{dualityEqregularity}), we arrive at
	\begin{equation*}
		\begin{split}
			\|\be_0\|^2&\leq Ch^2\|\bPhi\|_2\|\bu\|_2+C(\lambda+\mu)h^2\|\nabla\cdot \bPhi\|_1\|\bu\|_2+Ch^2\|\bbf\|_1\|\bPhi\|_2\\
			&\leq Ch^2(\|\bPhi\|_2+(\lambda+\mu)\|\nabla\cdot \bPhi\|_1)\|\bu\|_2+Ch^2\|\bbf\|_1\|\bPhi\|_2\\
			&\leq Ch^2(\|\bu\|_2+\|\bbf\|_1)\|\be_0\|,
		\end{split}
	\end{equation*}
	which completes the proof of Theorem.
	 
\end{proof}

\begin{remark}
	Based on the proofs of Theorem \ref{ehH1} and Theorem \ref{THL2}, it is evident that the boundedness of $\lambda\|\nabla\cdot\bu\|_1$ is not utilized in the proofs.
	Therefore, it can be concluded that the new WG Algorithm \ref{algo-primal} remains locking-free  even when $\lambda\|\nabla\cdot\bu\|_1$ is unbounded.
\end{remark}

\section{Numerical Results}\label{section:Numex}
In this section, numerical examples will be given to demonstrate the convergence and locking-free properties of the new WG Algorithm \ref{algo-primal}.
Consider the linear elasticity problems (\ref{primal_model})-(\ref{bc2}) in the two dimensional
square domain $\Omega =(0,1)^2$  and the three dimensional
unit cube $\Omega =(0,1)^3$.

Let $\bu \in [H^2(\Omega)]^d $ and $\bu_h=\{\bu_0,\bu_b\} \in V_h $ be the solution of linear elasticity problems (\ref{primal_model})-(\ref{bc2}) and the WG scheme (\ref{WGA_primal}), respectively. The error function $\textbf{e}_h$ is defined by
$$\be_h=\{\be_0,\be_b\}=\{Q_0\bu-\bu_0,
Q_b\bu-\bu_b\}.$$
Similar to the traditional finite element method, we consider two error norms as defined below:
\begin{eqnarray*}
	\3bar\be_h\3bar^2&=&\sum_{T\in\mathcal{T}_h}\left\{\int_{T}|\nabla_w\be_h|^2dT+h_T^{-1}\int_{\partial T}|Q_b\be_0-\be_b|^2ds\right\},
	\\
	\|\be_0\|^2&=&\sum_{T\in\mathcal{T}_h}\int_T|\be_0|^2dT.
\end{eqnarray*}

\subsection{Two dimensional convergence test}
Consider the square
domain $\Omega =(0,1)^2$ partitioned into uniform triangular grids $\T_h$ with meshsize $h$.
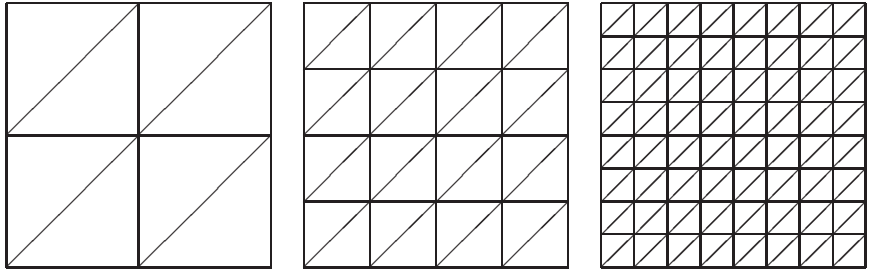
\begin{figure}[h!]
	\begin{center} \setlength\unitlength{1.25 pt}
		\begin{picture}(260,80)(0,0)
			\def\tr{\begin{picture}(20,20)(0,0)\put(0,0){\line(1,0){20}}\put(0,20){\line(1,0){20}}
					\put(0,0){\line(0,1){20}} \put(20,0){\line(0,1){20}} \put(0,0){\line(1,1){20}}
			\end{picture}}
			{\setlength\unitlength{2.5 pt}
				\multiput(0,0)(20,0){2}{\multiput(0,0)(0,20){2}{\tr}}}
			
			{\setlength\unitlength{1.25 pt}
				\multiput(90,0)(20,0){4}{\multiput(0,0)(0,20){4}{\tr}}}
			{\setlength\unitlength{0.625 pt}
				\multiput(360,0)(20,0){8}{\multiput(0,0)(0,20){8}{\tr}}}
	\end{picture}\end{center}
	\caption{The uniform triangular grids in two dimensional test problems with $h=\frac{1}{2},\frac{1}{4},\frac{1}{8}$.}
	\label{2Dlevel}
\end{figure}

\begin{example}\label{example2D1}
	Consider the  elasticity problems (\ref{primal_model})-(\ref{bc2}) with the exact solution
	\begin{eqnarray*}
		\bu=\left(\begin{array}{ccc}
			\sin(\pi x) \sin(\pi y) \\
			\sin(\pi x) \sin(\pi y)
		\end{array}\right).
	\end{eqnarray*}
	Set the $Lam\acute{e}$ constants $\mu=1$,~$\lambda=1$.
\end{example}

The numerical results of the Example \ref{example2D1} are shown in Table \ref{tab:notations1}. As can be seen, numerical results have achieved the optimal order of convergence, which is consistent with the theoretical results.

\begin{table}[!ht]
	\centering
	\caption{ Error and convergence order of displacement $\bu$ in Example \ref{example2D1}.}
	\label{tab:notations1}
	\begin{tabular}{ccccc}
		\hline
		1/h&$\3bar\bu_h-Q_h\bu\3bar$&order&$\|\bu_0-Q_0\bu\|$&order\\
		\hline
		8&6.7697e-01&--&7.1002e-02&--\\
		16&3.4495e-01&0.9727&1.8097e-02&1.9721\\
		32&1.7331e-01&0.9931&4.5467e-03&1.9929\\
		64&8.6759e-02&0.9983&1.1381e-03&1.9982\\
		128&4.3393e-02&0.9996&2.8461e-04&1.9996\\
		\hline
	\end{tabular}
\end{table}

\subsection{Three dimensional convergence test}
Consider the unit cube domain $\Omega =(0,1)^3$. We use tetrahedral meshes as shown in Figure \ref{level}.
\begin{figure}[h!]
	\centering
	\includegraphics[width=1 \columnwidth,height=0.4\linewidth]{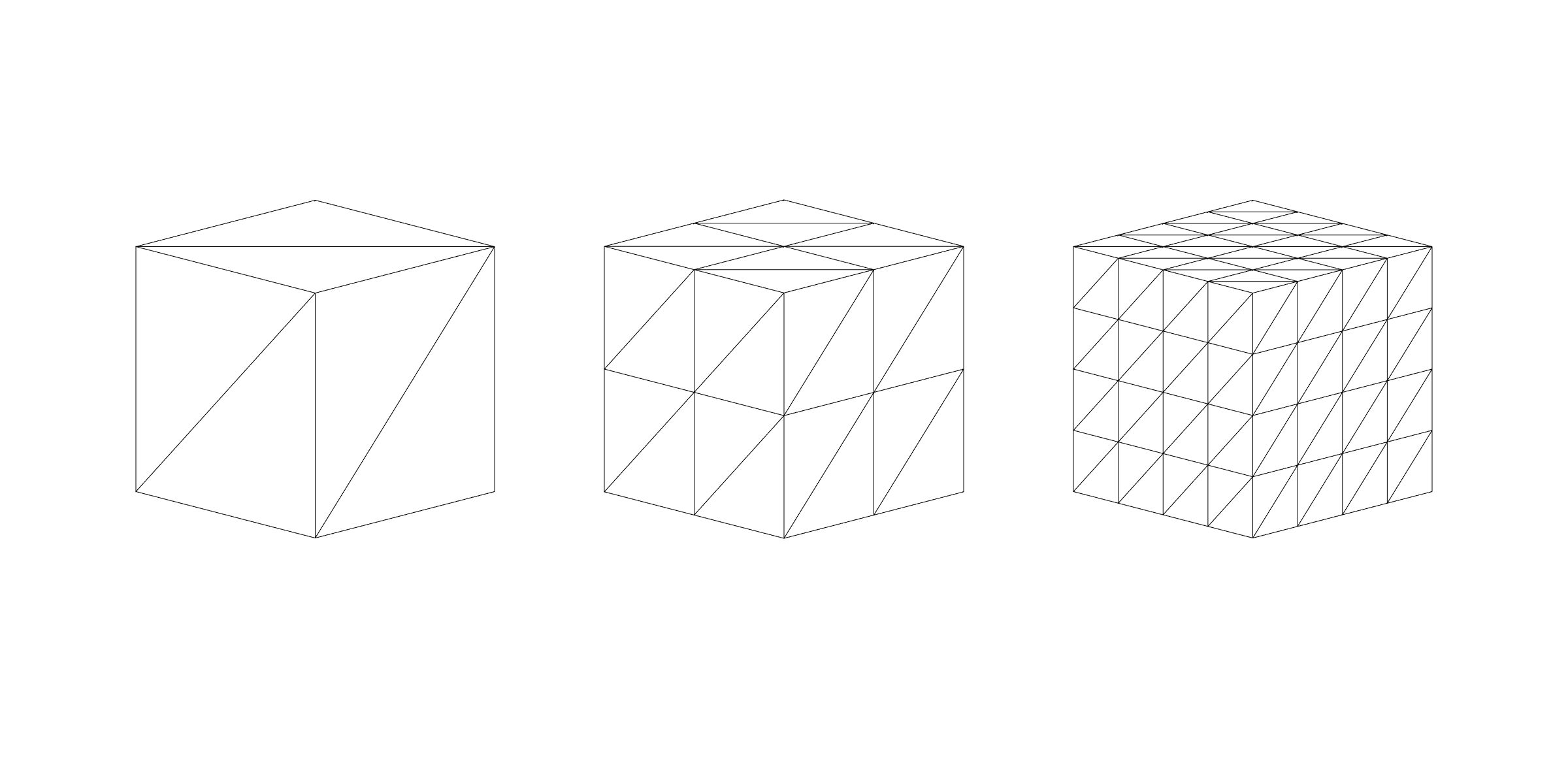}\\
	\vspace*{-4mm}
	\caption{The first three levels of tetrahedral grids in three dimensional test problems.}
	\label{level}
\end{figure}

\begin{example}\label{example3D4.1}
	Consider the elasticity problems (\ref{primal_model})-(\ref{bc2})  with $Lam\acute{e}$ constants $\mu=1$, $\lambda=1$ and the exact solution
	\begin{eqnarray*}
		\bu=\left(\begin{array}{ccc}
			\sin x \sin y \sin z\\
			\cos x \cos y \cos z\\
			\cos x \sin y \sin z
		\end{array}\right).
	\end{eqnarray*}
\end{example}

The error and convergence order are listed in Table \ref{tab:example3D4.1}. From the Table \ref{tab:example3D4.1}, it can be seen that in the three dimensional case, the numerical examples also achieve the optimal convergence order, which is consistent with the theory.

\begin{table}[h!]
	\centering
	\caption{Error and convergence order of displacement $\bu$ in Example \ref{example3D4.1}.}
	\label{tab:example3D4.1}
	\begin{tabular}{ccccc}
		\hline
		Level&$\3bar\bu_h-Q_h\bu\3bar$&order&$\|\bu_0-Q_0\bu\|$&order\\
		\hline
		2&2.4459e-01&--&3.9102e-02&--\\
		3&1.4957e-01&0.7096&1.4597e-02&1.4216\\
		4&8.2381e-02&0.8604&4.5213e-03&1.6909\\
		5&4.2773e-02&0.9456&1.2310e-03&1.8769\\
		6&2.1664e-02&0.9814&3.1692e-04&1.9577\\
		\hline
	\end{tabular}
\end{table}

\subsection{Locking test}
In this experiment, we validate the locking-free property of the new WG Algorithm \ref{algo-primal}.

\begin{example}\label{example2D2}(2D Locking-free test)
	Consider the elasticity problems (\ref{primal_model})-(\ref{bc2}) in the square
	domain $\Omega =(0,1)^2$. In this example, we use uniform triangular grids $\T_h$ with meshsize $h$, and set $Lam\acute{e}$ constant $\mu=1$. The exact solution $\bu$ is chosen as follows
	\begin{eqnarray*}
		\bu=\left( \begin{aligned}
			& (x^4-2x^3+x^2)(5y^4-8y^3+3y^2) \\
			& -(4x^3-6x^2+2x)(y^5-2y^4+y^3) \\
		\end{aligned} \right)+\dfrac{1}{\lambda}\left( \begin{aligned}
			& \sin(\pi x)\sin(\pi y) \\
			& \sin(\pi x)\sin(\pi y)\\
		\end{aligned} \right).
	\end{eqnarray*}
\end{example}

Tables \ref{tab:notations4}-\ref{tab:notations8} list the numerical results of $Lam\acute{e}$ constant $\lambda$ with different values.

\begin{table}[H]
	\centering
	\caption{Error and convergence order of displacement $\bu$ when $Lam\acute{e}$ constant $\lambda= 1$ in Example \ref {example2D2}.}
	\label{tab:notations4}
	\begin{tabular}{ccccc}
		\hline
		1/h&$\3bar\bu_h-Q_h\bu\3bar$&order&$\|\bu_0-Q_0\bu\|$&order\\
		\hline
		8&6.7787e-01&--&7.1208e-02&--\\
		16&3.4535e-01&0.9729&1.8137e-02&1.9731\\
		32&1.7350e-01&0.9931&4.5551e-03&1.9933\\
		64&8.6855e-02&0.9983&1.4000e-03&1.9984\\
		128&4.3441e-02&0.9996&2.8508e-04&1.9996\\
		\hline
	\end{tabular}
\end{table}

\begin{table}[H]
	\centering
	\caption{Error and convergence order of displacement $\bu$ when $Lam\acute{e}$ constant $\lambda= 10^2$ in Example \ref {example2D2}.}
	\label{tab:notations41}
	\begin{tabular}{ccccc}
		\hline
		1/h&$\3bar\bu_h-Q_h\bu\3bar$&order&$\|\bu_0-Q_0\bu\|$&order\\
		\hline
		8&1.8810e-02&--&2.0340e-03&--\\
		16&9.8071e-03&0.9396&5.3310e-04&1.9318\\
		32&4.9611e-03&0.9832&1.3556e-04&1.9755\\
		64&2.4881e-03&0.9956&3.4091e-05&1.9914\\
		128&1.2450e-03&0.9989&8.5418e-06&1.9968\\
		\hline
	\end{tabular}
\end{table}
\begin{table}[H]
	\centering
	\caption{Error and convergence order of displacement $\bu$ when $Lam\acute{e}$ constant $\lambda= 10^4$ in Example \ref {example2D2}.}
	\label{tab:notations414}
	\begin{tabular}{ccccc}
		\hline
		1/h&$\3bar\bu_h-Q_h\bu\3bar$&order&$\|\bu_0-Q_0\bu\|$&order\\
		\hline
		8&1.7263e-02&--&1.8492e-03&--\\
		16&9.0709e-03&0.9284&4.9503e-04&1.9013\\
		32&4.5986e-03&0.9801&1.2694e-04&1.9634\\
		64&2.3076e-03&0.9948&3.2036e-05&1.9864\\
		128&1.1548e-03&0.9987&8.0390e-06&1.9946\\
		\hline
	\end{tabular}
\end{table}

\begin{table}[H]
	\centering
	\caption{Error and convergence order of displacement $\bu$ when $Lam\acute{e}$ constant $\lambda= 10^6$ in Example \ref {example2D2}.}
	\label{tab:notations71}
	\begin{tabular}{ccccc}
		\hline
		1/h&$\3bar\bu_h-Q_h\bu\3bar$&order&$\|\bu_0-Q_0\bu\|$&order\\
		\hline
		8&1.7260e-02&--&1.8486e-03&--\\
		16&9.0697e-03&0.9283&4.9497e-04&1.9010\\
		32&4.5981e-03&0.9800&1.2693e-04&1.9633\\
		64&2.3073e-03&0.9948&3.2035e-05&1.9864\\
		128&1.1547e-03&0.9987&8.0390e-06&1.9946\\
		\hline
	\end{tabular}
\end{table}

\begin{table}[H]
	\centering
	\caption{Error and convergence order of displacement $\bu$ when $Lam\acute{e}$ constant $\lambda= 10^{8}$ in Example \ref {example2D2}.}
	\label{tab:notations8}
	\begin{tabular}{ccccc}
		\hline
		1/h&$\3bar\bu_h-Q_h\bu\3bar$&order&$\|\bu_0-Q_0\bu\|$&order\\
		\hline
		8&1.7260e-02&--&1.8486e-03&--\\
		16&9.0697e-03&0.9283&4.9497e-04&1.9010\\
		32&4.5981e-03&0.9800&1.2693e-04&1.9633\\
		64&2.3073e-03&0.9948&3.2035e-05&1.9863\\
		128&1.1547e-03&0.9987&8.0396e-06&1.9945\\
		\hline
	\end{tabular}
\end{table}

\begin{example}\label{example3D4.3}(3D Locking-free test)
	Consider the elasticity problems (\ref{primal_model})-(\ref{bc2}) with  $\Omega =(0,1)^3$. In the computation, we set $Lam\acute{e}$ constant $\mu=1$. The exact solution $\bu$ is chosen as follows
	\begin{eqnarray*}
		\bu=\left( \begin{aligned}
			&z^3 \sin x \sin y \\
			& 5z^3 \cos x \cos y \\
			& z^4 \cos x \sin y
		\end{aligned} \right)+\dfrac{1}{\lambda}\left( \begin{aligned}
			& \sin x \\
			& \sin y\\
			& \sin z
		\end{aligned} \right),
	\end{eqnarray*}
	where
	\begin{equation*}
		\nabla\cdot \bu=\dfrac{1}{\lambda}(\cos x+\cos y+\cos z).
	\end{equation*}
	In this example, we use tetrahedral grids as shown in Figure \ref{level}.
\end{example}

In the computation, we repeat the computation in Example \ref{example3D4.3}, with $Lam\acute{e}$ constant $\lambda=1$, $\lambda=10^2$, $\lambda=10^4$, $\lambda=10^6$, and $\lambda=10^{8}$. The numerical results are listed in Tables \ref{tab:example3D4.3.1}-\ref{tab:example3D4.3.3}.

\begin{table}[!ht]
	\centering
	\caption{Error and convergence order of displacement $\bu$ when $Lam\acute{e}$ constant $\lambda= 1$ in Example \ref {example3D4.3}.}
	\label{tab:example3D4.3.1}
	\begin{tabular}{ccccc}
		\hline
		Level&$\3bar\bu_h-Q_h\bu\3bar$&order&$\|\bu_0-Q_0\bu\|$&order\\
		\hline
		2     & 1.4977e+00 & -- & 2.1295e-01 &-- \\
		3     & 9.5175e-01 & 0.6541 & 7.2238e-02 & 1.5597 \\
		4     & 5.2024e-01 & 0.8714 & 2.0280e-02 & 1.8327 \\
		5     & 2.6805e-01 & 0.9567 & 5.2706e-03 & 1.9440 \\
		6     & 1.3531e-01 & 0.9863 & 1.3334e-03 & 1.9828 \\
		\hline
	\end{tabular}
\end{table}

\begin{table}[!ht]
	\centering
	\caption{Error and convergence order of displacement $\bu$ when $Lam\acute{e}$ constant $\lambda= 10^2$ in Example \ref {example3D4.3}.}
	\label{tab:example3D4.3.21}
	\begin{tabular}{ccccc}
		\hline
		Level&$\3bar\bu_h-Q_h\bu\3bar$&order&$\|\bu_0-Q_0\bu\|$&order\\
		\hline
		2     & 1.4607e+00 & -- & 2.2674e-01 & -- \\
		3     & 9.3918e-01 & 0.6372 & 8.5329e-02 & 1.4099 \\
		4     & 5.2231e-01 & 0.8465 & 2.6046e-02 & 1.7120 \\
		5     & 2.7164e-01 & 0.9432 & 7.0264e-03 & 1.8902 \\
		6     & 1.3760e-01 & 0.9812 & 1.8004e-03 & 1.9645 \\
		\hline
	\end{tabular}
\end{table}

\begin{table}[!ht]
	\centering
	\caption{Error and convergence order of displacement $\bu$ when $Lam\acute{e}$ constant $\lambda= 10^4$ in Example \ref {example3D4.3}.}
	\label{tab:example3D4.3.22}
	\begin{tabular}{ccccc}
		\hline
		Level&$\3bar\bu_h-Q_h\bu\3bar$&order&$\|\bu_0-Q_0\bu\|$&order\\
		\hline
		2     & 1.4609e+00 & -- & 2.2764e-01 & -- \\
		3     & 9.3930e-01 & 0.6372 & 8.6275e-02 & 1.3998 \\
		4     & 5.2240e-01 & 0.8464 & 2.6533e-02 & 1.7012 \\
		5     & 2.7169e-01 & 0.9432 & 7.1903e-03 & 1.8836 \\
		6     & 1.3763e-01 & 0.9812 & 1.8464e-03 & 1.9613 \\
		\hline
	\end{tabular}
\end{table}

\begin{table}[!ht]
	\centering
	\caption{Error and convergence order of displacement $\bu$ when $Lam\acute{e}$ constant $\lambda= 10^6$ in Example \ref {example3D4.3}.}
	\label{tab:example3D4.3.2}
	\begin{tabular}{ccccc}
		\hline
		Level&$\3bar\bu_h-Q_h\bu\3bar$&order&$\|\bu_0-Q_0\bu\|$&order\\
		\hline
		2     & 1.4609e+00 & -- & 2.2765e-01 & -- \\
		3     & 9.3930e-01 & 0.6372 & 8.6285e-02 & 1.3997 \\
		4     & 5.2240e-01 & 0.8464 & 2.6538e-02 & 1.7011 \\
		5     & 2.7169e-01 & 0.9432 & 7.1921e-03 & 1.8836 \\
		6     & 1.3763e-01 & 0.9812 & 1.8469e-03 & 1.9613 \\
		\hline
	\end{tabular}
\end{table}

\begin{table}[!ht]
	\centering
	\caption{Error and convergence order of displacement $\bu$ when $Lam\acute{e}$ constant $\lambda= 10^{8}$ in Example \ref {example3D4.3}.}
	\label{tab:example3D4.3.3}
	\begin{tabular}{ccccc}
		\hline
		Level&$\3bar\bu_h-Q_h\bu\3bar$&order&$\|\bu_0-Q_0\bu\|$&order\\
		\hline
		2     & 1.4609e+00 & -- & 2.2765e-01 & -- \\
		3     & 9.3930e-01 & 0.6372 & 8.6285e-02 & 1.3997 \\
		4     & 5.2240e-01 & 0.8464 & 2.6538e-02 & 1.7011 \\
		5     & 2.7169e-01 & 0.9432 & 7.1921e-03 & 1.8836 \\
		6     & 1.3763e-01 & 0.9812 & 1.8470e-03 & 1.9612 \\
		\hline
	\end{tabular}
\end{table}

From  Tables \ref{tab:notations4}-\ref{tab:notations8} and Tables \ref{tab:example3D4.3.1}-\ref{tab:example3D4.3.3}, it can be seen that the errors are not affected by $\lambda$, which implies that the WG numerical scheme is locking-free. This is consistent with the previous theoretical analysis.

In the following example,  we consider the situation when the  term $\lambda\|\nabla\cdot\bu\|_1$ is unbounded.

\begin{example}\label{example2D6}(2D Locking-free test with unbounded $\lambda\|\nabla\cdot\bu\|_1$)
	Consider the elasticity problems (\ref{primal_model})-(\ref{bc2}) in the square
	domain $\Omega =(0,1)^2$. In this example, we use uniform triangular grids $\T_h$ with meshsize $h$, and set the $Lam\acute{e}$ constant $\mu=1$. The exact solution $\bu$ is chosen as follows
	\begin{eqnarray*}
		\bu=\left( \begin{aligned}
			& \sin(\pi x)\cos(\pi y) \\
			& \cos(\pi x)\sin(\pi y) \\
		\end{aligned} \right),
	\end{eqnarray*}
	and the right-hand side function $\bbf$ is
	\begin{eqnarray*}
		\bbf=-\mu\left( \begin{aligned}
			& -2\pi^2\sin(\pi x)\cos(\pi y) \\
			& -2\pi^2\cos(\pi x)\sin(\pi y) \\
		\end{aligned} \right)-(\lambda+\mu)\left( \begin{aligned}
			& -2\pi^2\sin(\pi x)\cos(\pi y) \\
			& -2\pi^2\cos(\pi x)\sin(\pi y) \\
		\end{aligned} \right).
	\end{eqnarray*}
\end{example}

Table \ref{tab:example2D61} illustrates the error and convergence order of the new  WG Algorithm  \ref{algo-primal} and the standard WG Algorithm \ref{algo-primal_old} for various values of the $Lam\acute{e}$ constant $\lambda$. 
From the table, it can be seen that the  errors from the standard WG Algorithm \ref{algo-primal_old} depend on the $Lam\acute{e}$ constant $\lambda$.  As the value of $Lam\acute{e}$ constant $\lambda$ increases, both the energy norm and the $L^2$ norm of the approximate error increase. Different from the standard WG Algorithm \ref{algo-primal_old}, the displacement errors of the new  WG Algorithm  \ref{algo-primal} are independent of the $Lam\acute{e}$ constant $\lambda$, indicating that the new  WG Algorithm  \ref{algo-primal} is robust with respect to $\lambda$.

\begin{table}[!ht]
	\centering
	\caption{Error and convergence order of displacement $\bu$ in Example \ref {example2D6}.}
	\label{tab:example2D61}
	\scalebox{0.90}{
		\begin{tabular}{ccccccccc}
			\hline
			& \multicolumn{4}{c}{New WG Algorithm \ref{algo-primal}} & \multicolumn{4}{c}{Standard WG Algorithm \ref{algo-primal_old}} \\
			\hline
			1/h&$\3bar\bu_h-Q_h\bu\3bar$&order&$\|\bu_0-Q_0\bu\|$&order&$\3bar\bu_h-Q_h\bu\3bar$&order&$\|\bu_0-Q_0\bu\|$&order\\
			\hline
			\multicolumn{9}{c}{$\lambda=1$}\\
			\hline
			8     & 4.2490e-01 & --& 4.4184e-02 &  --  & 2.5235e+00 & --   & 1.5532e-01 & --\\
			16    & 2.2787e-01 & 0.8989 & 1.1832e-02 & 1.9008 & 1.2748e+00 & 0.9851 & 3.8405e-02 & 2.0150 \\
			32    & 1.1607e-01 & 0.9732 & 3.0130e-03 & 1.9735 & 6.3929e-01 & 0.9957 & 9.5883e-03 & 2.0020 \\
			64    & 5.8312e-02 & 0.9932 & 7.5679e-04 & 1.9932 & 3.1989e-01 & 0.9989 & 2.3965e-03 & 2.0003 \\
			128   & 2.9190e-02 & 0.9983 & 1.8942e-04 & 1.9983 & 1.5998e-01 & 0.9997 & 5.9910e-04 & 2.0001 \\
			\hline
			\multicolumn{9}{c}{$\lambda=10^2$}\\
			\hline
			8     & 4.1003e-01 & --  & 4.3290e-02 &   -- & 9.3427e+01 &  --  & 6.3728e+00 &--\\
			16    & 2.2518e-01 & 0.8647 & 1.1878e-02 & 1.8658 & 4.7794e+01 & 0.9670 & 1.6276e+00 & 1.9692 \\
			32    & 1.1562e-01 & 0.9617 & 3.0507e-03 & 1.9611 & 2.4067e+01 & 0.9897 & 4.1074e-01 & 1.9865 \\
			64    & 5.8211e-02 & 0.9900 & 7.6811e-04 & 1.9898 & 1.2056e+01 & 0.9972 & 1.0297e-01 & 1.9961 \\
			128   & 2.9157e-02 & 0.9975 & 1.9237e-04 & 1.9974 & 6.0311e+00 & 0.9993 & 2.5760e-02 & 1.9990 \\
			
			\hline
			\multicolumn{9}{c}{$\lambda=10^4$}\\
			\hline
			8     & 4.1000e-01 &   --  & 4.3330e-02 & --  & 9.1867e+03 &  -- & 6.2943e+02 & -- \\
			16    & 2.2517e-01 & 0.8646 & 1.1894e-02 & 1.8651 & 4.7010e+03 & 0.96659 & 1.6089e+02 & 1.9679 \\
			32    & 1.1562e-01 & 0.9617 & 3.0555e-03 & 1.9608 & 2.3674e+03 & 0.98963 & 4.0616e+01 & 1.9860 \\
			64    & 5.8210e-02 & 0.9900 & 7.6935e-04 & 1.9897 & 1.1860e+03 & 0.99722 & 1.0183e+01 & 1.9959 \\
			128   & 2.9157e-02 & 0.9975 & 1.9269e-04 & 1.9974 & 5.9329e+02 & 0.99929 & 2.5476e+00 & 1.9989 \\
			\hline
			\multicolumn{9}{c}{$\lambda=10^6$}\\
			\hline
			8     & 4.1000e-01 & --  & 4.3331e-02 &   -- & 9.1851e+05 & --   & 6.2935e+04 & -- \\
			16    & 2.2517e-01 & 0.8646 & 1.1894e-02 & 1.8651 & 4.7002e+05 & 0.96658 & 1.6088e+04 & 1.9679 \\
			32    & 1.1562e-01 & 0.9617 & 3.0555e-03 & 1.9608 & 2.3670e+05 & 0.98963 & 4.0611e+03 & 1.9860 \\
			64    & 5.8210e-02 & 0.9900 & 7.6937e-04 & 1.9897 & 1.1858e+05 & 0.99722 & 1.0182e+03 & 1.9959 \\
			128   & 2.9157e-02 & 0.9975 & 1.9269e-04 & 1.9974 & 5.9319e+04 & 0.99929 & 2.5473e+02 & 1.9989 \\
			\hline
			\multicolumn{9}{c}{$\lambda=10^8$}\\
			\hline
			8  & 4.1000e-01 & --   & 4.3331e-02 & --   & 9.1851e+07 & --  & 6.2935e+06 & --\\
			16    & 2.2517e-01 & 0.8646 & 1.1894e-02 & 1.8651 & 4.7002e+07 & 0.9665 & 1.6088e+06 & 1.9679 \\
			32    & 1.1562e-01 & 0.9617& 3.0555e-03 & 1.9608 & 2.3670e+07 & 0.9896 & 4.0611e+05 & 1.9860 \\
			64    & 5.8210e-02 & 0.9900 & 7.6937e-04 & 1.9897 & 1.1858e+07 & 0.9972 & 1.0182e+05 & 1.9959 \\
			128   & 2.9157e-02 & 0.9975 & 1.9271e-04 & 1.9973 & 5.9319e+06 & 0.9992 & 2.5473e+04 & 1.9989 \\
			\hline
		\end{tabular}
	}
\end{table}

\section*{Acknowledgments}
Y. Wang was supported by the the National Natural Science Foundation of China (grant No. 12171244). R. Zhang and R. Wang were supported by the National Natural Science Foundation of China (grant No. 11971198, 12001230), the National Key Research and Development Program of China (grant No. 2020YFA0713602), and the Key Laboratory of Symbolic Computation and Knowledge Engineering of Ministry of Education of China housed at Jilin University.

\bibliographystyle{siam}
\bibliography{lib}

\end{document}